\documentclass[
final
]{dmtcs-episciences}

\usepackage[utf8]{inputenc}
\usepackage{amssymb, url}
\usepackage{amsthm}
\usepackage{amsmath}
\usepackage{verbatim}
\usepackage{graphicx}
\usepackage{subcaption}
\usepackage{setspace}
\usepackage{enumerate}
\usepackage{float}
\usepackage{tikz}
\usetikzlibrary{shapes,arrows,spy,positioning,decorations,calc}%,snakes}
\usepackage[marginal,splitrule,multiple]{footmisc}
\usepackage{listings}

\newcounter{results}[section]

\newtheorem{thm}[results]{Theorem}

\newtheorem{lem}[results]{Lemma}

\newtheorem{conj}[results]{Conjecture}

\newtheorem{ques}[results]{Question}

\DeclareMathOperator{\dg}{deg}

\newcommand{\flr}[1]{\left\lfloor #1 \right\rfloor}
\newcommand{\cel}[1]{\left\lceil #1 \right\rceil}

\newcommand{\R}{\mathbb{R}}

\author{
	Zhanar Berikkyzy\affiliationmark{1}
  \and Axel Brandt\affiliationmark{2}
  \and Sogol Jahanbekam\affiliationmark{3}\\
  \and Victor Larsen\affiliationmark{4}
  \and Danny Rorabaugh\affiliationmark{5}}
\title[List-antimagic labeling of vertex-weighted graphs]{List-antimagic labeling of vertex-weighted graphs\thanks{Research supported in part by NSF grant DMS-1427526, ``The Rocky Mountain--Great Plains Graduate Research Workshop in Combinatorics".}}

\affiliation{
  Fairfield University, CT, USA\\
  Northern Kentucky University, KY, USA\\
  San Jos\'e State University, CA, USA\\
  University of Sioux Falls, SD, USA\\
  Knoxville, TN, USA}
\received{2019-07-17}
\revised{2021-06-16}
\accepted{2021-07-19}
\begin{document}
\publicationdetails{23}{2021}{3}{5}{5631}
\maketitle
\begin{abstract}
A graph $G$ is \emph{weighted--$k$--list--antimagic} if for any vertex weighting $\omega\colon V(G)\to\mathbb{R}$ and any list assignment $L\colon E(G)\to2^{\mathbb{R}}$ with $|L(e)|\geq |E(G)|+k$ there exists an edge labeling $f$ such that $f(e)\in L(e)$ for all $e\in E(G)$, labels of edges are pairwise distinct, and the sum of the labels on edges incident to a vertex plus the weight of that vertex is distinct from the sum at every other vertex.
In this paper we prove that every graph on $n$ vertices having no $K_1$ or $K_2$ component is weighted--$\left\lfloor{\frac{4n}{3}}\right\rfloor$--list--antimagic.
\end{abstract}

%%%%%%%%%%%%%%%%%%%%%%%%%%%%%%%%%%%%%%%%%%%%%
%%%%%%%%%%%%%%%%%%%%%%%%%%%%%%%%%%%%%	INTRO
%%%%%%%%%%%%%%%%%%%%%%%%%%%%%%%%%%%%%%%%%%%%%
\section{Introduction} \label{sec:intro}
In this paper we consider simple, finite graphs. In an edge-labeling of a graph $G$, we define the \textit{vertex sum} at a vertex $v$ to be the sum of labels of edges incident to $v$. 
A graph $G$ is \textit{antimagic} if there exists a bijective edge labeling from $E(G)$ to $\{1,\dotsc,|E(G)|\}$ such that the vertex sums are pairwise distinct. 
%Graphs are a fundamental abstraction for studying any sort of network or system of related objects.  Graph coloring and labeling problems naturally arise when the objects modeled by vertices in a graph are divided into different classes or when the associations modeled by edges are qualitatively or quantitatively distinct.  One might consider antimagic labelings in a situation that necessitates the vertices to be locally distinguishable. 

The concept of antimagic graphs was first introduced by Hartsfield and Ringel in \cite{HR90}.
Excluding $K_2$, they proved that cycles, paths, complete graphs, and wheels are antimagic and they made the following conjecture:
\begin{conj}[\cite{HR90}] \label{conj:AM}
Every simple connected graph other than $K_2$ is antimagic.
\end{conj}

More than 25 years later, this conjecture remains open.
Progress has been made for various minimum, maximum, and average degree conditions and for regular graphs.
%The most significant work toward proving this conjecture is by 
Alon et al. \cite{AKLRY04} proved that there is a constant $C$ such that every graph with $n$ vertices and minimum degree at least $C\log{n}$ is antimagic.
They also proved that a graph $G$ on $n$ vertices is antimagic if $n\geq4$ and $\Delta(G)\geq n-2$.
Yilma \cite{Y13} improved this bound from $n-2$ to $n-3$ for $n\geq9$.
Later, Eccles \cite{E15} proved the conjecture for graphs with average degree at least 4182.
Further developments have focused on regular graphs, Cranston et al. \cite{CLZ15} proved that $k$-regular graphs are antimagic, when $k$ is odd $k\geq3$.
This was followed by B\'{e}rczi et al. \cite{BBV15} and Chang et al. \cite{CLPZ16} proving that $k$-regular graphs are antimagic, when $k$ is even and $k\geq4$.
Following partial results of Deng and Li \cite{DL19} and Lozano et al. \cite{LMS19}, Lozano et al. \cite{LMST21} proved that all caterpillars are antimagic.
%Still the conjecture remains open for many classes of graphs

Due to the elusiveness of the original conjecture, several notions have been considered as either a measure of closeness to being antimagic or a variation thereof.
We direct the interested reader to Gallian's dynamic survey of graph labeling \cite{G13} for a more thorough discussion.
%Our motivation for this paper is to apply a powerful tool which has been rarely applied to labeling problems, the Combinatorial Nullstellensatz, in order to improve the results on $k$-weighted antimagic graphs, which we define below.  
%Further, we discuss how these techniques can strengthen results for antimagic labeling of oriented graphs, which demonstrates the versatility of Combinatorial Nullstellensatz as a tool for labeling problems.

A graph $G$ is called \textit{$k$--antimagic} if there exists an injective edge labeling from $E(G)$ into the set \mbox{$\{1,\dotsc,|E(G)|+k\}$} such that vertex sums are pairwise distinct. 
Note that antimagic is equivalent to $0$--antimagic. 
If for any vertex weighting $\omega:V(G)\to\R$, there exists a bijective edge labeling $\phi: E(G) \to \{1,\dotsc,|E(G)|\}$ such that $\omega(u)+\sum_{ux\in E(G)}\phi(ux)\neq \omega(v)+\sum_{vx\in E(G)} \phi(vx)$ for all $u,v\in V(G)$, then $G$ is called \textit{weighted--antimagic}.
We call $\omega(v)+\sum_{vx\in E(G)} \phi(vx)$ the \textit{weighted vertex sum} at vertex $v$.
When a graph is described using a combination of variations in this paper, it satisfies the conditions of each variation mentioned in its description.
For example, a graph $G$ is called \textit{weighted--$k$--antimagic} if for any vertex weighting from $V(G)$ into $\R$, there exists an injective edge labeling from $E(G)$ into $\{1,\dotsc,|E(G)|+k\}$ such that weighted vertex sums are pairwise distinct. 

The argument used in the previously mentioned result of Alon et al. \cite{AKLRY04} extends to show that every graph $G$ with minimum degree at least $C\log|V(G)|$ is weighted--0--antimagic.
However, there are connected graphs that are not weighted--0--antimagic; for example, $K_{1,n}$.
Further, Wong and Zhu \cite{WZ12} provided a family of connected graphs with even number of vertices that is not weighted--1--antimagic.
In investigating the natural question of finding the smallest integer $k$ for which a graph is weighted--$k$--antimagic, Wong and Zhu posed the following questions: %:
%\begin{ques}[\cite{WZ12}] \label{ques:WZ1}
Is it true that every connected graph other than $K_2$ is weighted--$2$--antimagic? 
%\end{ques}
%\begin{ques}[\cite{WZ12}] \label{ques:WZ2}
Is there a connected graph $G$ with an odd number of vertices which is not weighted--$1$--antimagic? 
%\end{ques}
%\noindent 
Improving upon a result of Hefetz in \cite{H05} showing that every connected graph other than $K_2$ is weighted--$(2|V(G)|-4)$--antimagic, Wong and Zhu also proved the following: 
\begin{thm} [\cite{WZ12}] \label{thm:wAM32}
Every connected graph on $n$ vertices with $n\geq 3$ is weighted--$\left(\cel{\frac{3n}{2}}-2\right)$--antimagic.
\end{thm}

The main result of this paper improves upon this result by lowering $\left(\cel{\frac{3n}{2}}-2\right)$, including disconnected graphs in the result, and proving the results for list-coloring.
To this end, a graph $G$ is \textit{$k$--list--antimagic} if for any list assignment $L\colon E(G)\to 2^{\R}$, where $|L(e)|\geq|E(G)|+k$ for all $e\in E(G)$, there exists an edge labeling that assigns each edge $e$ a label from $L(e)$ such that edge labels are pairwise distinct and vertex sums are pairwise distinct.
%We improve upon Theorem \ref{thm:wAM32} by proving the following broader theorem:
\begin{thm} \label{thm:lwAM43}
Every graph on $n$ vertices with no $K_1$ or $K_2$ component is weighted--$\flr{\frac{4n}{3}}$--list--antimagic.
\end{thm}

We prove Theorem~\ref{thm:lwAM43}, our main result, in Section~\ref{sec:undir}.
With minor modifications, the proof can be used for antimagic labelings of oriented graphs, a variant introduced by Hefetz, M\"utze, and Schwartz \cite{HMS10}.
Oriented graphs are briefly discussed in Section~\ref{sec:orient}. 

%Introduced in \cite{HMS10}, an oriented graph $G$ is \textit{oriented--antimagic} if there exists a bijective edge labeling from $E(G)$ to $\{1,\dotsc,|E(G)|\}$ such that oriented vertex sums are pairwise distinct, where an \textit{oriented vertex sum} at a vertex $v$ is the sum of labels of edges incident to and oriented toward $v$ minus the sum of labels of edges incident to and oriented away from $v$. 
%An \textit{orientation of $G$} is a directed graph with $G$ as its underlying graph.
%
%Hefetz, M\"utze, and Schwartz \cite{HMS10} prove that there is a constant $C$ such that every orientation of a graph on $n$ vertices with minimum degree at least $C\log n$ is oriented--antimagic. 
%They also show that every orientation of complete graphs, wheels, stars with at least 4 vertices, and regular graphs of odd degree are oriented--antimagic. 
%In addition, they show that every regular graph on $n$ vertices with even degree and a matching of size $\flr{\frac{n}{2}}$ has an orientation that is oriented--antimagic.
%They make the following conjecture and ask the subsequent question.
%\begin{conj}[\cite{HMS10}] \label{conj:oAM}
%Every connected undirected graph admits an orientation that is oriented--antimagic.
%\end{conj} 
%\begin{ques}[\cite{HMS10}] \label{ques:oAM}
%Is every connected oriented graph on at least 4 vertices oriented--antimagic?
%\end{ques}
%\noindent Toward Conjecture \ref{conj:oAM}, we prove the following.
%\begin{thm} \label{thm:oAM23}
%Every graph on $n$ vertices admits an orientation that is $\flr{\frac{2n}{3}}$--oriented--antimagic.
%\end{thm}

Before proving our results in Section~\ref{sec:undir}, we present some useful tools.
The primary tool used in the results is the Combinatorial Nullstellensatz.
\begin{thm}[Combinatorial Nullstellensatz, \cite{A99}] \label{CN}
Let $\mathbb{F}$ be an arbitrary field, and let \linebreak $f=f(x_1,\ldots, x_n)$ be a polynomial in $\mathbb{F}[x_1,\ldots,x_n]$. Suppose the degree $d(f)$ of $f$ is $\sum_{i=1}^n t_i$, where each $t_i$ is a nonnegative integer, and suppose the coefficient of $\prod_{i=1}^n x_i^{t_i}$ in $f$ is nonzero. Then, if $S_1,\ldots,S_n$ are subsets of $\mathbb{F}$ with $|S_i|>t_i$, there are $s_1\in S_1$, $s_2\in S_2$, $\ldots$, $s_n\in S_n$ so that $f(s_1,\ldots,s_n)\neq 0$.
\end{thm}

A useful lemma when applying the Combinatorial Nullstellensatz is Equation (5.16) from \cite{FGIL94}, which is stated below.
\begin{lem}[\cite{FGIL94}] \label{lem:coef}
The coefficient of the monomial
$\displaystyle\prod_{1\leq i\leq N} x_i^{s(N-1)+i-1}$ 
in the polynomial \\
$\displaystyle\prod_{1\leq i<j\leq N} (x_i-x_j)^{2s+1}$
has absolute value $\frac{\big((s+1)N\big)!}{N!(s+1)!^N}$.
\end{lem}
\noindent Note that the polynomial in the above lemma is the determinant of the $(2s+1)^{st}$ power of the Vandermonde matrix.

%%%%%%%%%%%%%%%%%%%%%%%%%%%%%%%%%%%%%%%%%%%%
%%%%%%%%%%%%%%%%%%%%%%%%%%%%%%%%%	Undirected
%%%%%%%%%%%%%%%%%%%%%%%%%%%%%%%%%%%%%%%%%%%%
\section{Antimagic Results} \label{sec:undir}
The main results of this paper rely on an inductive argument that has the potential to create isolated vertices or $K_2$ components.
Since the creation of these components would preclude an antimagic labeling, we define the following to account for this possibility.
A graph $G$ is \textit{$k$--quasi--antimagic} if there exists an injective edge labeling from $E(G)$ into $\{1,\dotsc,|E(G)|+k\}$ such that vertex sums are pairwise distinct for pairs of non--isolated vertices that are not adjacent in a $K_2$ component.
Notice that if every component of a graph has at least 3 vertices, $k$--quasi--antimagic is equivalent to $k$--antimagic.

Throughout the proof  we denote a vertex of degree at least $j$ in a graph $G$ by a \textit{$j^+$--vertex}.
An \textit{even (odd) component} in a graph is a component that has an even (odd) number of vertices.
A vertex $v$ is in edge $e$, denoted $v\in e$, if $e$ is incident to $v$.
We use notation from \cite{W96} unless otherwise specified.

The following lemma provides the basis step of our inductive argument:

\begin{lem} \label{lem:maxd2}
If $G$ is a graph on $n$ vertices and $\Delta(G) \leq 2$,
then $G$ is weighted--$\flr{\frac{4n}{3}}$--list--quasi--antimagic.
\end{lem}

\begin{proof}
It suffices to prove the lemma for graphs with $\delta(G) \geq 1$, since adding isolated vertices increases $n$ without adding any additional labeling requirements.

Let $G$ have $m$ edges.
Given $1 \leq \delta(G) \leq \Delta(G) \leq 2$, every component of $G$ is a path or cycle and has at least $2$ vertices.
Let $e_1, \ldots, e_q$ be the $q$ isolated edges of $G$, $D_1, \ldots, D_r$ be the $r$ even components of $G$ each having at least $4$ vertices, and $C_1, \ldots, C_s$ be the odd components of $G$.
Let $\omega\colon V(G)\to \R$ be a vertex weighting and $L\colon E(G)\to2^{\R}$ be a list function such that $|L(e)| \geq m+\flr{\frac{4n}{3}}$ for all $e\in E(G)$.

Let $E'=\{e_1,\ldots, e_k\}$ be a matching in $G$ of maximum size.
Notice that $e_1,\ldots,e_q$ are in $E'$, so $k\geq q$.
In fact, counting gives $k = \frac{n-s}{2}$.
Also define $v_i$ for each $i \in \{1,2,\ldots,s\}$ to be the unique vertex in $C_i$ such that $v_i$ is not incident to any edge in $E'$.
Let $E''=E(G)-E'$.
Note that $|E''| = m-\frac{n - s}{2}$.

\begin{figure}[!htb]
\centering
\includegraphics[width=\textwidth]{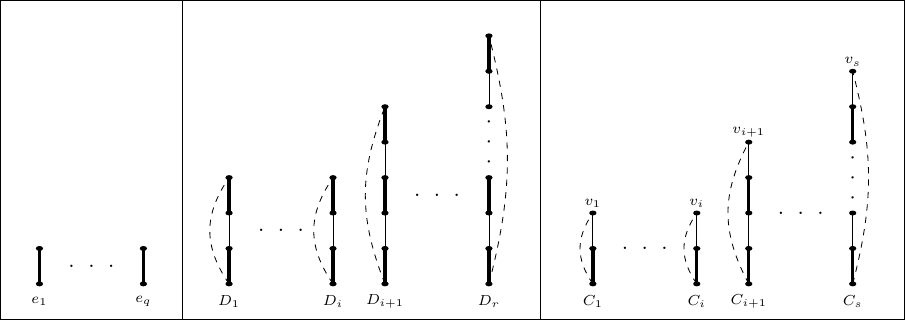}
\caption{Components (paths or cycles) of $G$ with edges in a maximum matching $E'$  in bold.}
\label{fig:E}
\end{figure}

In the first stage of this proof, we iteratively label the edges of $E''$ as follows.
For edge $e=yz \in E''$, we label $e$ from $L(e)$ so that
(1) the label assigned to $e$ is not already assigned to an edge,
(2) the weighted vertex sum of $u\in\{y,z\}$ is not equal to the weighted vertex sum of $w\in N(u)\setminus\{y,z\}$, and
(3) if $e$ is incident to some $v_i$, the weighted vertex sum at $v_i$ is distinct from the weighted vertex sum at $v_j$ for each $j \neq i$.
With these three restrictions, there are at most $(|E''| -1) + 2 + (s - 1)$ values that are not allowed when labeling each edge in $E''$.
Since $s \leq \frac{n}{3}$, we have
\[ |E''| + s = m-\frac{n- s}{2} + s = m+ \frac{3s}{2} -\frac{n}{2} \leq m <m+\flr{\frac{4n}{3}}.\]
Since $|L(e)|\geq m+\flr{\frac{4n}{3}}$ for each $e\in E''$, there are more labels on each edge than possible restrictions.
Therefore, such a labeling on $E''$ is possible.

The second stage of this proof is to label the edges of the maximum matching $E'$ in $G$.
Let $f''\colon E'' \rightarrow \R$ be the partial edge labeling described above and define $\omega''\colon V(G)\to \R$ to be $\omega''(v)=\omega(v)+\sum_{vx\in E''} f''(vx)$.
Note that $\omega''(v_1),\ldots,\omega''(v_s)$ are distinct and are not impacted by labeling edges in $E'$.
Also, if $uv\in E'$ is not an isolated edge, then $\omega''(u)\neq\omega''(v)$ because of the labeling of $E''$.
We construct a polynomial with variables $x_1,\ldots,x_k$ corresponding to the label of $e_1,\ldots,e_k$, respectively.
Equal edge labels and equal vertex sums appear in $G$ precisely at zeroes of the polynomial
\begin{align*}
g(x_1, \ldots, x_k) &= \prod_{1\leq i < j \leq k} \phi(i,j) \times \prod_{1\leq i\leq k} \psi(i), \mbox{ where}\\
\phi(i,j)&=(x_i-x_j)\prod_{\substack{u \in e_i\\u' \in e_j}}(x_i+\omega''(u)-x_j-\omega''(u')), \mbox{ and}\\
\psi(i)&=\prod_{e \in E''} (x_i - f''(e)) \prod_{1\leq j \leq s}\prod_{u \in e_i} (x_i + \omega''(u) - \omega''(v_j)).
\end{align*}
One can check that, for $1\leq i<j\leq k$, $\phi(i,j)=0$ if and only if either $e_i$ and $e_j$ have been given the same labels or the final vertex sum of an endpoint of $e_i$ matches the final vertex sum of an endpoint of $e_j$.
Also, for $1\leq i\leq k$, $\psi(i)=0$ if and only if the label $x_i$ is already used in $E''$ or one endpoint of $e_i$ has the same final vertex sum as $v_j$ for any $j\in\{1,2,\ldots,s\}$.
Note that the maximum degree in $g$ is at most $5{k \choose 2} + k(2s + m - k)$; we will show this is the maximum degree in $g$ below. 

The monomials of $g$ with maximum degree have the same coefficients as they do in polynomial
\[ h(x_1, \ldots, x_k) = \prod_{1\leq i < j \leq k}(x_i-x_j)^5 \prod_{1\leq i\leq k} x_i^{2s + m - k}.\]
By Lemma \ref{lem:coef}, the monomial 
\[x_1^{2(k-1) + (2s+m-k)}x_2^{2(k-1) + 1 + (2s+m-k)}x_3^{2(k-1) + 2 + (2s+m-k)}\dotsm x_k^{3(k-1) + (2s+m-k)}\]
has nonzero coefficient in $h$ and thus also in $g$.
Recall that $k = \frac{n-s}{2}$ and $s \leq \frac{n}{3}$. 
Hence
\begin{align*}
3(k-1) + (2s+m-k) &= m + n + s - 3\\
 	&\leq m + n + \flr{\frac{n}{3}} - 3\\
	&< m + \flr{\frac{4n}{3}}\!\!.
\end{align*}
Since $|L(e_i)|\geq m+\flr{\frac{4n}{3}}$ for all $e_i$, Theorem \ref{CN} implies that labels $x_1,\ldots,x_k$ can be chosen so that $g(x_1,\ldots,x_k)\neq0$.
By the construction of $g$, this implies that $G$ has a weighted--$\flr{\frac{4n}{3}}$--list--quasi--antimagic labeling.
\end{proof}

\begin{lem} \label{lem:red3vtx}
Let $G$ be an $n$-vertex graph that is not weighted--$\flr{\frac{4n}{3}}$--list--quasi--antimagic. Suppose that $G$ has the fewest edges of any graph with this property.  Then $\Delta(G)<3$.
\end{lem}

\begin{proof}
Let $G$ be an edge--minimal graph on $m$ edges with $\Delta(G)\geq 3$ that is not weighted--$\flr{\frac{4n}{3}}$--list--quasi--antimagic. Let $\omega\colon V(G)\to\R$ and $L\colon E(G)\to 2^{\R}$ such that $|L(e)|\geq m+\flr{\frac{4n}{3}}$ for all $e\in E(G)$.

Suppose that $v$ is a $3^+$--vertex with neighbors $u_1$, $u_2$, and $u_3$. 
Let $G'=G-\{vu_1,vu_2,vu_3\}$. By the choice of $G$, $G'$ is weighted--$\flr{\frac{4n}{3}}$--list--quasi--antimagic. Thus there is a labeling $f$ of $E(G')$ using labels in the lists of its edges that is a weighted--$\flr{\frac{4n}{3}}$--list--quasi--antimagic labeling of $G'$.
We apply the Combinatorial Nullstellensatz to extend $f$ to an edge labeling of $G$ which is weighted--$\flr{\frac{4n}{3}}$--list--quasi--antimagic.

Let $x_1$,  $x_2$, and $x_3$ correspond to the labels of edges $vu_1$, $vu_2$, and $vu_3$, respectively. 
For each $w\in V(G')$, let $\omega''(w)$ denote the weighted vertex sum of $w$ in $G'$. 
We define the following polynomial in which respective factors ensure a distinct edge labeling for the edges $vu_1,vu_2,vu_3$, distinct weighted sums for any pair between $V(G)-\{v,u_1,u_2,u_3\}$ and $\{v,u_1,u_2,u_3\}$, any pair between $v$ and $\{u_1,u_2,u_3\}$, and any pair in $\{u_1,u_2,u_3\}$: 
\begin{align*}
g(x_1,x_2,x_3) =& \prod_{1\leq i<j\leq3} (x_i-x_j) 
				\prod_{w\notin \{v,u_1,u_2,u_3\}} (\omega''(v)+x_1+x_2+x_3-\omega''(w))\\
			&\times \prod_{i=1}^3 \prod_{w\notin \{v,u_1,u_2,u_3\}} (x_i+\omega''(u_i)-\omega''(w))\\
			&\times \prod_{i=1}^3 (\omega''(v)+x_1+x_2+x_3-x_i-\omega''(u_i)) \\
			& \times \prod_{1\leq i <j\leq 3} (x_i+\omega''(u_i)-x_j-\omega''(u_j)).
\end{align*}
By construction, $g(x_1,x_2,x_3)=0$ when $x_i\in L(vu_i)-\{f(e)\colon e\in E(G')\}$ if and only if labels chosen for $x_1$, $x_2$, and $x_3$ do not create a weighted--$\flr{\frac{4n}{3}}$--list--quasi--antimagic labeling.
Note that \[\dg(g)= {3 \choose 2}+(n-4)+3(n-4)+3+{3\choose2}=4n-7.\]
Therefore the coefficient of any monomial $x_1^ax_2^bx_3^c$ in $g$, where $a+b+c=4n-7$, is the same as its coefficient in the polynomial
\begin{align*}
g'(x_1,x_2,x_3)=x_1^{n-4}x_2^{n-4}x_3^{n-4}(x_1+x_2+x_3)^{n-4}
		\prod_{1\leq i<j\leq3}(x_i-x_j)^2(x_i+x_j).
\end{align*}
To use Theorem \ref{CN}, we would like a fairly balanced triple $(a,b,c)$ with $a+b+c=4n-7$ such that the coefficient of $x_1^ax_2^bx_3^c$ in $g'$ (and thus in $g$) is nonzero.  It suffices to find $(a',b',c')$ with $a'+b'+c'=n+5$ so that the coefficient of $x_1^{a'}x_2^{b'}x_3^{c'}$ is nonzero in the polynomial
\begin{align*}
h(x_1,x_2,x_3)=(x_1+x_2+x_3)^{n-4}
		\prod_{1\leq i<j\leq3}(x_i-x_j)^2(x_i+x_j).
\end{align*}
For some $i\in\{0,1,2\}$, we can write $n+5=3k+i$.  We use the notation $\left[x_1^{a'}x_2^{b'}x_3^{c'}\right]_h$ to refer to the coefficient of $x_1^{a'}x_2^{b'}x_3^{c'}$ in the polynomial $h$.  Defining $a'=k+i+1$, $b'=k$, and $c'=k-1$,
\begin{align*}
\left[x_1^{a'}x_2^{b'}x_3^{c'}\right]_h=\sum_{\alpha+\beta+\gamma=9}{n-4 \choose a'-\alpha,b'-\beta,c'-\gamma}\left[x_1^\alpha x_2^\beta x_3^\gamma\right]_{\hat{h}},
\end{align*}
where $\hat{h}$ is the polynomial
\begin{align*}
\hat{h}(x_1,x_2,x_3)=\prod_{1\leq i< j\leq3}(x_i-x_j)^2(x_i+x_j).
\end{align*}

We verify that this coefficient is nonzero for $n\geq3$ in Appendix \ref{APPEND}.  
%The Sage code \cite{sage} used to verify that this coefficient is nonzero for $k > 2$ and any $i$ (and thus for $n>3$) is found in an auxiliary file.
Therefore the corresponding coefficient $\left[x_1^{a}x_2^{b}x_3^{c}\right]_g$ is also nonzero where $\max\{a,b,c\}=a=a'+(n-4)=\frac{4n-4+2i}{3}\leq\flr{\frac{4n}{3}}$.

Define $L'(vu_i)=L(vu_i)-\{f(e)\colon e\in E(G')\}$.
Since $|L(vu_i)|\geq m+\flr{\frac{4n}{3}}$, we have $|L'(vu_i)|\geq\flr{\frac{4n}{3}}+3$.
Thus, by Theorem \ref{CN}, there are labels $f(vu_1)$, $f(vu_2)$, and $f(vu_3)$ in $L'(vu_1)$, $L'(vu_2)$, and $L'(vu_3)$, respectively, for which $g(f(vu_1),f(vu_2),f(vu_3))$ is nonzero.
Therefore we obtain a weighted--$\flr{\frac{4n}{3}}$--list--quasi--antimagic labeling of $G$, contradicting the choice of $G$.
\end{proof}

Theorem \ref{thm:lwAM43} follows from the following result, which is a direct result of the contradiction between Lemmas \ref{lem:maxd2} and \ref{lem:red3vtx}.

\begin{thm}\label{thm:lwqAM}
Every graph on $n$ vertices is weighted--$\flr{\frac{4n}{3}}$--list--quasi--antimagic.
\end{thm}

%\noindent Taking a similar approach to that in Lemma \ref{lem:red3vtx} may be advantageous in showing that a graph on $n$ vertices with maximum degree $d$ is $\flr{\frac{(d+1)n}{d}}$--$\ell$wqAM.

\noindent \textbf{Remark:} 
A generalized version of Lemma \ref{lem:red3vtx} might claim that $\Delta(G)<d$ when $G$ is not weighted--$\flr{\frac{(d+1)n}{d}}$--list--quasi--antimagic. 
A modification to the computation in Appendix~\ref{APPEND} confirms that this general form holds for $d=4,5$.
%However, this generalization does not lead to an improvement of Theorem \ref{thm:lwqAM} because the technique of Lemma \ref{lem:maxd2} does not extend beyond $\Delta(G)\leq2$.
However, the technique of Lemma \ref{lem:maxd2} does not extend beyond $\Delta(G)\leq2$.
As such, improving Theorem \ref{thm:lwqAM} is left as an area for future investigation.

%%%%%%%%%%%%%%%%%%%%%%%%%%%%%%%%%%%%%%%%%%%%%
%%%%%%%%%%%%%%%%%%%%%%%%%%%%%%%%%%%%%	Oriented
%%%%%%%%%%%%%%%%%%%%%%%%%%%%%%%%%%%%%%%%%%%%%
\section{Remarks on Oriented Graphs} \label{sec:orient}
%Introduced in \cite{HMS10}, a
An oriented graph $G$ is \textit{oriented--antimagic} if there exists a bijective edge labeling from $E(G)$ to \linebreak $\{1,\dotsc,|E(G)|\}$ such that oriented vertex sums are pairwise distinct, where an \textit{oriented vertex sum} at a vertex $v$ is the sum of labels of all edges entering $v$ minus the sum of labels of all edges leaving $v$. 
An \textit{orientation} of $G$ is a directed graph with $G$ as its underlying graph.

Hefetz, M\"utze, and Schwartz \cite{HMS10} proved that there is a constant $C$ such that every orientation of a graph on $n$ vertices with minimum degree at least $C\log n$ is oriented--antimagic. 
They also showed that every orientation of complete graphs, wheels, stars with at least 4 vertices, and regular graphs of odd degree are oriented--antimagic. 
In addition, they showed that every regular graph on $n$ vertices with even degree and a matching of size $\flr{\frac{n}{2}}$ has an orientation that is oriented--antimagic.
They made the following conjecture and asked the subsequent question:
\begin{conj}[\cite{HMS10}] \label{conj:oAM}
Every connected undirected graph admits an orientation that is oriented--antimagic.
\end{conj} 
\begin{ques}[\cite{HMS10}] \label{ques:oAM}
Is every connected oriented graph on at least 4 vertices oriented--antimagic?
\end{ques}
\noindent Recently, a variety of papers  proved that various graph classes admit an antimagic orientation, see \linebreak \cite{LSWY19,LMS19,SY17,SH19,SYZ21,Y19}.  

Our approach to proving Theorem \ref{thm:lwAM43} can be modified slightly for showing that graphs admit a $k$--antimagic orientation.
An oriented edge labeled with a non-zero value contributes differently to the vertex sums of its two incident vertices. 
Thus, an exception is no longer necessary for isolated edges.  %.
That is, for weighted graphs, the only difference between quasi--antimagic and antimagic is that the former allows multiple isolated vertices.
Moreover, %(if we ignore the list and weighted generalizations)
the freedom to choose the orientation of every edge doubles the effectiveness of our use of Combinatorial Nullstellensatz in making progress toward Conjecture~\ref{conj:oAM}. 
Essentially, the ability to change sign means that we can have $\flr{\frac{4n}{3}}$ elements in the sets $T_i$ required by Theorem \ref{CN} by only including $\flr{\frac{2n}{3}}$ extra values.
Indeed, the natural %necessary 
modifications to the polynomials in the proofs of Lemmas~\ref{lem:maxd2} and~\ref{lem:red3vtx} give that every graph on $n$ vertices (with at most one isolated vertex) admits an orientation that is $\flr{\frac{2n}{3}}$--oriented--antimagic. 

\section{Acknowledgements}
The authors would like to thank the organizers of the 2014 Rocky Mountain--Great Plains Graduate Research Workshop in Combinatorics without whom this collaboration would not have been possible.
They would also like to acknowledge Nathan Graber, Kirsten Hogenson, and Lauren M. Nelsen, workshop participants who contributed in the initial exploration of this problem.

\nocite{*}
%\bibliographystyle{alpha}
% use the following instead if you encounter problems 
\bibliographystyle{alpha}
%\bibliography{refs}

\newcommand{\etalchar}[1]{$^{#1}$}

\appendix 
\section{Appendix} \label{APPEND}
Notice that the polynomial $\hat{h}$ from Lemma~\ref{lem:red3vtx} expands to
\begin{align*}
\hat{h}(x_1,x_2,x_3)
&=x_1^6 x_2^3 - x_1^6 x_2^2 x_3 - x_1^6 x_2 x_3^2 + x_1^6 x_3^3 - x_1^5 x_2^4 + 2 x_1^5 x_2^2 x_3^2 - x_1^5 x_3^4 - x_1^4 x_2^5+ 2 x_1^4 x_2^4 x_3\\
&  - x_1^4 x_2^3 x_3^2 - x_1^4 x_2^2 x_3^3 + 2 x_1^4 x_2 x_3^4 - x_1^4 x_3^5 + x_1^3 x_2^6 - x_1^3 x_2^4 x_3^2 - x_1^3 x_2^2 x_3^4 + x_1^3 x_3^6 \\
&- x_1^2 x_2^6 x_3+ 2 x_1^2 x_2^5 x_3^2 - x_1^2 x_2^4 x_3^3 - x_1^2 x_2^3 x_3^4 + 2 x_1^2 x_2^2 x_3^5 - x_1^2 x_2 x_3^6 - x_1 x_2^6 x_3^2 \\
&+ 2 x_1 x_2^4 x_3^4 - x_1 x_2^2 x_3^6 + x_2^6 x_3^3 - x_2^5 x_3^4 - x_2^4 x_3^5 + x_2^3 x_3^6.
\end{align*}

\noindent Thus
%---- Full/Original Version
\begin{align*}\label{Coefficient}
[&x_1^{a'}x_2^{b'}x_3^{c'}]_h= \sum_{\alpha+\beta+\gamma=9}{{n-4}\choose{k+i+1-\alpha,k-\beta,k-1-\gamma}}[x_1^{\alpha}x_2^{\beta}x_3^{\gamma}]_{\hat{h}}\\
&={{n-4}\choose{k+i+1-6,k-3,k-1-0}}-{{n-4}\choose{k+i+1-6,k-2,k-1-1}}\\
&-{{n-4}\choose{k+i+1-6,k-1,k-1-2}}+{{n-4}\choose{k+i+1-6,k-0,k-1-3}}\\
&-{{n-4}\choose{k+i+1-5,k-4,k-1-0}}+2{{n-4}\choose{k+i+1-5,k-2,k-1-2}}\\
&-{{n-4}\choose{k+i+1-5,k-0,k-1-4}}-{{n-4}\choose{k+i+1-4,k-5,k-1-0}}\\
&+2{{n-4}\choose{k+i+1-4,k-4,k-1-1}}-{{n-4}\choose{k+i+1-4,k-3,k-1-2}}\\
&-{{n-4}\choose{k+i+1-4,k-2,k-1-3}}+2{{n-4}\choose{k+i+1-4,k-1,k-1-4}}\\
&-{{n-4}\choose{k+i+1-4,k-0,k-1-5}}+{{n-4}\choose{k+i+1-3,k-6,k-1-0}}\\
&-{{n-4}\choose{k+i+1-3,k-4,k-1-2}}-{{n-4}\choose{k+i+1-3,k-2,k-1-4}}\\
&+{{n-4}\choose{k+i+1-3,k-0,k-1-6}}-{{n-4}\choose{k+i+1-2,k-6,k-1-1}}\\
&+2{{n-4}\choose{k+i+1-2,k-5,k-1-2}}-{{n-4}\choose{k+i+1-2,k-4,k-1-3}}\\
&-{{n-4}\choose{k+i+1-2,k-3,k-1-4}}+2{{n-4}\choose{k+i+1-2,k-2,k-1-5}}\\
&-{{n-4}\choose{k+i+1-2,k-1,k-1-6}}-{{n-4}\choose{k+i+1-1,k-6,k-1-2}}\\
&+2{{n-4}\choose{k+i+1-1,k-4,k-1-4}}-{{n-4}\choose{k+i+1-1,k-2,k-1-6}}\\
&+{{n-4}\choose{k+i+1-0,k-6,k-1-3}}-{{n-4}\choose{k+i+1-0,k-5,k-1-4}}\\
&-{{n-4}\choose{k+i+1-0,k-4,k-1-5}}+{{n-4}\choose{k+i+1-0,k-3,k-1-6}}.
\end{align*}
%---- Simplified/Canceled Version
%\begin{align*}\label{Coefficient}
%[x_1^{a'}x_2^{b'}x_3^{c'}]_h&= \sum_{\alpha+\beta+\gamma=9}{{n-4}\choose{k+i+1-\alpha,k-\beta,k-1-\gamma}}[x_1^{\alpha}x_2^{\beta}x_3^{\gamma}]_{\hat{h}}\\
%&=-{{n-4}\choose{k+i+1-6,k-2,k-1-1}}+{{n-4}\choose{k+i+1-6,k-0,k-1-3}}\\
%&-{{n-4}\choose{k+i+1-5,k-4,k-1-0}}+2{{n-4}\choose{k+i+1-5,k-2,k-1-2}}\\
%&-{{n-4}\choose{k+i+1-5,k-0,k-1-4}}+{{n-4}\choose{k+i+1-4,k-4,k-1-1}}\\
%&-{{n-4}\choose{k+i+1-4,k-3,k-1-2}}+{{n-4}\choose{k+i+1-4,k-1,k-1-4}}\\
%&-{{n-4}\choose{k+i+1-4,k-0,k-1-5}}+{{n-4}\choose{k+i+1-3,k-6,k-1-0}}\\
%&-{{n-4}\choose{k+i+1-3,k-4,k-1-2}}-{{n-4}\choose{k+i+1-3,k-2,k-1-4}}\\
%&+{{n-4}\choose{k+i+1-3,k-0,k-1-6}}-{{n-4}\choose{k+i+1-2,k-6,k-1-1}}\\
%&+{{n-4}\choose{k+i+1-2,k-5,k-1-2}}-{{n-4}\choose{k+i+1-2,k-4,k-1-3}}\\
%&+2{{n-4}\choose{k+i+1-2,k-2,k-1-5}}-{{n-4}\choose{k+i+1-2,k-1,k-1-6}}\\
%&-{{n-4}\choose{k+i+1-1,k-6,k-1-2}}+2{{n-4}\choose{k+i+1-1,k-4,k-1-4}}\\
%&-{{n-4}\choose{k+i+1-1,k-2,k-1-6}}+{{n-4}\choose{k+i+1-0,k-6,k-1-3}}\\
%&-{{n-4}\choose{k+i+1-0,k-5,k-1-4}}-{{n-4}\choose{k+i+1-0,k-4,k-1-5}}\\
%&+{{n-4}\choose{k+i+1-0,k-3,k-1-6}}.
%\end{align*}

For convenience, we refer to the above expression above as the \textit{Coefficient} and aim to show that it is nonzero.
We aim to show the Coefficient is nonzero for all $n\geq3$ and proceed to consider the three cases for $n+5=3k+i$, namely $i=0,1,2$.

\noindent\underline{Case $i=0$:}
Setting $i=0$ in the Coefficient and factoring ${{n-4}\choose{k-3,k-3,k-3}}$ from each term gives 
\begin{align*}
&-2\frac{(k-3)(k-4)}{(k-2)(k-2)}-2\frac{(k-3)^2}{(k-1)(k-2)}+2\frac{k-3}{k-2}-1-2\frac{(k-3)(k-4)(k-5)}{k(k-1)(k-2)}\\
&+2\frac{(k-3)(k-4)}{(k-1)(k-2)}+2\frac{(k-3)(k-4)(k-5)}{(k-1)(k-2)^2}-\frac{(k-3)(k-4)(k-5)(k-6)}{(k-1)^2(k-2)^2}\\
&+2\frac{(k-3)^2(k-4)}{k(k-1)(k-2)}-\frac{(k-3)^2(k-4)^2}{(k+1)k(k-1)(k-2)}+\frac{(k-3)(k-4)(k-5)(k-6)}{(k+1)k(k-1)(k-2)},
\end{align*}
which simplifies to
\[\frac{-96k^3+672k^2-1344k+528}{(k+1)k(k-1)^2(k-2)^2}.\]
To show the Coefficient is not zero, first note that the multinomial ${{n-4}\choose{k-3,k-3,k-3}}$ and the common denominator are positive integers for all $k\geq 3$.
It suffices to show that the numerator is not zero for $k\geq3$.
The derivative of the numerator with respect to $k$ implies that the numerator is decreasing for $k>\frac{7+\sqrt{7}}{3}\approx 3.22$.
Since the numerator is $-48$ and $-240$ when $k=3,4$ respectively, the numerator is negative for $k\geq 3$.

\noindent\underline{Case $i=1$:}
Setting $i=1$ in the Coefficient and factoring ${{n-4}\choose{k-2,k-3,k-3}}$ from each term gives
\begin{align*}
&-2\frac{k-3}{k-1}+1+\frac{(k-3)(k-4)(k-5)}{(k-1)^2(k-2)}-\frac{(k-3)(k-4)(k-5)}{(k+1)k(k-1)}+2\frac{(k-3)^2(k-4)}{(k+1)k(k-1)}\\
&-\frac{(k-3)(k-4)(k-5)(k-6)}{(k+1)k(k-1)(k-2)}-\frac{(k-3)^2(k-4)^2}{(k+2)(k+1)k(k-1)}+\frac{(k-3)(k-4)(k-5)(k-6)}{(k+2)(k+1)k(k-1)},
\end{align*}
which simplifies to
\[\frac{-96k^3+864k^2-2016k+1104}{(k+2)(k+1)k(k-1)^2(k-2)}\]
As before, the multinomial ${{n-4}\choose{k-2,k-3,k-3}}$ and the common denominator are positive integers for all $k\geq3$.
It suffices to show that the numerator is not zero for $k\geq3$.
The derivative of the numerator with respect to $k$ implies that the numerator is decreasing for $k>3+\sqrt{2}\approx 4.41$. 
Since the numerator is 240, 720, 624, and $-624$ for $k=3,4,5,6$ respectively, the numerator is negative for $k\geq6$ and the desired outcome holds.

\noindent\underline{Case $i=2$:}
Setting $i=2$ in the Coefficient and factoring ${{n-4}\choose{k-2,k-2,k-3}}$ from each term gives
\begin{align*}
&1-2\frac{(k-3)(k-4)}{k(k-1)}-\frac{k-2}{k-1}+\frac{(k-3)(k-4)}{(k-1)^2}+\frac{(k-3)(k-4)(k-5)(k-6)}{k^2(k-1)^2}\\
&+\frac{(k-2)(k-3)(k-4)}{(k+1)k(k-1)}-\frac{(k-2)(k-3)^2}{(k+1)k(k-1)}+\frac{(k-3)(k-4)(k-5)}{(k+1)k(k-1)}\\
&-\frac{(k-3)(k-4)(k-5)(k-6)}{(k+1)k(k-1)^2}-\frac{(k-2)(k-3)(k-4)(k-5)}{(k+2)(k+1)k(k-1)}\\
&+2\frac{(k-2)(k-3)^2(k-4)}{(k+2)(k+1)k(k-1)}-\frac{(k-3)(k-4)(k-5)(k-6)}{(k+2)(k+1)k(k-1)}\\
&-\frac{(k-2)(k-3)^2(k-4)^2}{(k+3)(k+2)(k+1)k(k-1)}+\frac{(k-2)(k-3)(k-4)(k-5)(k-6)}{(k+3)(k+2)(k+1)k(k-1)},
\end{align*}
which simplifies to
\[\frac{-96k^4+1248k^3-3456k^2+1728k+2160}{(k+3)(k+2)(k+1)k^2(k-1)^2}.\]
Again, the multinomial ${{n-4}\choose{k-2,k-2,k-3}}$ and the denominator are positive integers for all $k\geq3$.
The derivative of the numerator with respect to $k$ implies that the numerator is decreasing for $k\geq8$.
Since the numerator is 240, 2160, 9072, 20400, 33264, 42480, 40560, 17712, and $-38160$ for $k=2,3,\ldots,10$ respectively, the numerator is negative for $k\geq 10$ and the desired outcome holds.

As a result of these cases, the coefficient is nonzero for all $k\geq 3$ when $i=0,1$ and for all $k\geq 2$ when $i=2$.
This covers all possible values of $n\geq 3$ based on $n+5=3k+i$ for some $i=0,1,2$.

\end{document}